\documentclass[a4paper,reqno]{amsart}
\usepackage[english]{babel}
\usepackage{amsmath,amsthm,amssymb,amsfonts}

\usepackage{soul}
\usepackage{comment}
\usepackage{hyperref}
\usepackage{mathtools}
\usepackage[T1]{fontenc}
\usepackage[utf8]{inputenc}
\usepackage{dsfont}
\hypersetup{
 colorlinks,
 linkcolor={blue!90!black},
 citecolor={red!80!black},
 urlcolor={blue!50!black}
}
\usepackage{tikz}
\usepackage{tikz-cd}
\usepackage{graphicx}
\usepackage[all,cmtip]{xy}
\usepackage{mleftright}
\usepackage{booktabs}
\usepackage{cite}

\newtheorem{theorem}{\sc \textbf{Theorem}}[section]  
\newtheorem{proposition}[theorem]{\sc \textbf{Proposition}}   
        
\newtheorem{lemma}[theorem]{\sc \textbf{Lemma}}

\theoremstyle{remark}
\newtheorem{definition}[theorem]{\sc \textbf{Definition}}

\newtheorem{remark}[theorem]{\sc \textbf{Remark}}
\newtheorem{example}[theorem]{\sc \textbf{Example}}

\def\cK{{\mathcal K}}

\def\cM{{\mathcal M}}

\def\cO{{\mathcal O}}
\def\cP{{\mathcal P}}

\def\cU{{\mathcal U}}

\numberwithin{equation}{section}

\newcommand{\loc}{\mathrm{loc}}

\DeclareMathOperator*{\esssup}{ess\,sup}

\usepackage{pdfpages}

\makeatletter
\@namedef{subjclassname@1991}{Mathematics Subject Classification}
\usepackage{enumerate}
\graphicspath{{images/}}
\title[Carleson measures on locally finite trees]{Carleson measures on locally finite trees}
\author[A. Ottazzi]{Alessandro Ottazzi}
\address[A. Ottazzi]{School of Mathematics and Statistics \\  University of New South
Wales \\ 2052 Sydney  \\  Australia}
\email{a.ottazzi@unsw.edu.au}
\author[F. Santagati]{Federico Santagati}
\address[F. Santagati]{School of Mathematics and Statistics \\  University of New South
Wales \\ 2052 Sydney  \\  Australia}
\email{f.santagati@unsw.edu.au}

\thanks{The authors were both supported by the Australian Research Council, grant DP220100285. Santagati was partially supported by the INdAM - GNAMPA Project ``$L^p$ estimates for singular integrals in nondoubling settings'' (CUP E53C23001670001).
}
\keywords{Carleson measures, trees, Poisson integral, BMO}
\subjclass[2020]{05C05, 05C21,31C05, 43A99}
\begin{document} \maketitle
\begin{abstract}
We provide a characterization of Carleson measures on locally finite trees. This characterization establishes the connection between Carleson measures and the boundedness of a suitable Poisson integral between $L^p$-spaces. Additionally, when the tree has bounded degree, we investigate the relationship between Carleson measures and BMO functions defined on the boundary of the tree. 
\end{abstract}
\section{Introduction and preliminaries}
\subsection{Introduction}
Carleson measures were originally introduced and characterized in the Euclidean setting by L. Carleson \cite{Carl1,Carl2}. In this context, it was proved that a positive measure $\sigma$ is a Carleson measure if and only if the classical Poisson integral defines a bounded operator from $L^p(\mathbb R^n, \mathrm{d}x)$ to $L^p(\mathbb R^+\times \mathbb R^n, \sigma)$. \\  Subsequently, C. Fefferman and E. M. Stein established a connection between Carleson measures and functions of bounded mean oscillation on $\mathbb R^n$ ($BMO(\mathbb R^n)$). Indeed, they exhibited a suitable class of operators acting on $BMO(\mathbb R^n)$ such that the square of the image of a function is the density of a Carleson measure with respect to the Lebesgue measure \cite{FS}. 
These results have been generalized to other contexts. For instance, when $\mathbb{R}^n$ is replaced by a space of homogeneous type \cite{HaS}, or when $\mathbb{R}^+\times \mathbb{R}^n$ is replaced by a homogeneous tree (or more in general, by a radial tree), analogous results have been obtained \cite{CCS, CCS2}.     \smallskip \\ 
In this note, we aim to provide a characterization of Carleson measures on locally finite trees without any further restriction on the geometry of the tree. This is the content of Theorem \ref{th: Poisson}, which is proved in 
Section \ref{sec:2}. It is worth mentioning that a tree is not a product space; rather, it is the image of a surjective map 
on the product of the boundary of the tree itself and $\mathbb Z$. We use a natural definition of Carleson measures involving sectors instead of cylinders, which adapts to our setting the definition given in \cite{CCS, CCS2, CCPS}.
We point out that, because of the lack of symmetries in our settings, we cannot exploit the techniques used in the aforementioned papers. 
In particular, we introduce a suitable Laplacian associated with a random walk, where the probability of transitioning from a vertex to a neighbour is nonzero only if the neighbour lies below the original vertex. This hierarchical structure on the tree arises naturally from the choice of a root, which in this paper will be a point of the boundary of the tree. We then give a Poisson integral representation formula for harmonic functions associated with this distinguished Laplacian. In Section \ref{sec:2}, we introduce a  Hardy space $H^p$ containing harmonic functions, which serves as the discrete counterpart of the Hardy space $H^p(\mathbb{R}^+\times \mathbb{R})$ on the upper half-plane. It turns out that, when $p>1$, $H^p$ characterizes the space of harmonic functions which are Poisson integrals of $L^p$ functions on the boundary of the tree, see Theorem \ref{th: charactPoisson}. Subsequently, Carleson measures $\sigma$ are characterized as measures on the tree for which the Poisson integral maps continuously the natural $L^p$ space on the boundary of the tree to  $L^p(\sigma)$. This property is shown to be equivalent to the boundedness of the identity map from $H^p$ to $L^p(\sigma)$.\\ In Section \ref{sec:3}, we focus on trees with bounded degree and, in Theorem \ref{th: BMO},  we prove a result in the spirit of \cite{FS, HaS}, which relates Carleson measures to $BMO$ functions on the boundary of the tree.  We show that a class of integral operators whose kernels satisfy suitable cancellation and decay properties maps $BMO$ functions to functions that are densities of Carleson measures with respect to a suitable reference measure. Additionally, we prove that a converse statement holds true. In fact, Theorem \ref{th: BMO} provides a characterization of the $BMO$ space on the boundary of a tree. \smallskip \\ 
Throughout the paper, $C$ will denote a positive constant which may vary from line to line and that is independent of any involved variable but may depend on fixed parameters. Sometimes, we will stress such a dependence by adding a subscript.
\subsection{Preliminaries and notation} In this section we introduce the notation and recall some well-known results on trees. \\ 
Let $T$ be a tree and let $d$ denote the usual geodesic distance on $T$. We fix a root of the tree by choosing a point $\omega_*$ in the boundary of $T$, that we denote by $\Omega$ (see \cite[Chapter I.1]{FTN} for a detailed definition). We then introduce the {\it punctured boundary} $\partial T=\Omega \setminus \{\omega_*$\}. The choice of a root induces a partial order on $T$: given two vertices $x,y \in T$, we say that $x$ lies below $y$ (or equivalently, $y$ lies above $x$) if $y \in [x,\omega_*)$, where $[x,\omega_*)$ denotes the infinite geodesic starting from $x$ and ending in $\omega_*$. Similarly, we say that $\omega \in \partial T$ lies below the vertex $x$ if $x \in (\omega,\omega_*),$ that is $x$ belongs to the doubly infinite geodesic with endpoints $\omega$ and $\omega_*$.   We shall write $x \le y$ whenever $x \in T \cup \partial T$ lies below $y \in T$. \\ We say that a vertex $x$ is a neighbour of $y \in T$ if there is one edge connecting $x$ and $y$ or, equivalently, if $d(x,y)=1$. In this case, we write $x \sim y$. \\ We fix once and for all an origin $o \in T$ and we denote by $\{x_j\}_{j=0}$ an enumeration of the geodesic starting at $o$ and ending in $\omega_*$ such that $x_j\sim x_{j+1}$ for every $j\ge 0$. We define the level of a vertex by $$\ell(x)=\lim_{j \to \infty} j-d(x,x_j) \qquad \forall x \in T.$$  
 For every $x \in T$ we define the set of successors by
$$s(x)=\{y \sim x \ : \ \ell(y)=\ell(x)-1\}$$ and the predecessor of $x \in T$ {by} $$p(x)=\{y \sim x \ : \ \ell(y)=\ell(x)+1\}.$$ 
We set $p^0(x)=x$ and define inductively $p^n(x)=p(p^{n-1}(x))$ for every $x \in T$ and $n \ge 1$. Similarly, we set $s_0(x)=\{x\}$ and define \begin{align}s_n(x)=\cup_{y \in {s_{n-1}(x)}}s(y) \qquad \forall x \in T, n \ge 1.\end{align}
We say that a positive function $m$ on $T$ is a {\it flow measure} if it satisfies the conservation property \begin{align}\label{def:flow}
    m(x)=\sum_{y \in s(x)}m(y) \qquad \forall x \in T.
\end{align} Observe that by iterating \eqref{def:flow}, a flow measure also satisfies 
\begin{align*}
    m(x)=\sum_{y \in s_n(x)}m(y) \qquad \forall x \in T, n \in \mathbb N.
\end{align*}
Given two points $\eta,\zeta \in \overline{T}:=T \cup \partial T$,  their   confluent is
$$\eta \wedge \zeta=\arg\min\{\ell(x) \ :  x\in T, \eta,\zeta \le x\},$$
and for every $x \in T$, the sector $T_x$ and its boundary at infinity $\partial T_x$ are \begin{align*}
    T_x&=\{y \in T \ : y \le x\}, \\ 
    \partial T_x&=\{\omega \in \partial T \ : \omega \le x\}.
\end{align*} We define the Gromov distance on $\overline T \times \overline T$, denoted $\rho$, by
\begin{align*}
\rho(\eta,\xi)=\begin{cases}0 &\text{if $\eta=\xi$}, \\ e^{\ell(\eta\wedge \xi)} &\text{otherwise.}
    \end{cases}
\end{align*}
It is straightforward that the collection of balls with positive radii in $(\partial T, \rho)$ consists of $\{\partial T_x\}_{x \in T}$. Indeed, the Gromov distance of two points only depends on the level of their confluent. Moreover, it is clear that if $\partial T_y \cap \partial T_x \ne \emptyset$, then either $\partial T_x \subset \partial T_y$ or $\partial T_y \subset \partial T_x$. Analogous considerations hold if $\partial T_x$ and $\partial T_y$ are replaced by $T_x$ and $T_y$, respectively.\\ 
Observe that for every positive measure $\nu$ on $\partial T$ that is finite and positive on balls, there is an associated natural flow measure  given by 
\begin{align*}
    m_\nu(x)=\nu(\partial T_x) \qquad \forall x \in T.
\end{align*}   From now on, we shall always assume that a measure on $\partial T$ is finite and positive on balls. \\  
 Given a measure $\nu$ on $\partial T$ and a measure $m$ on $T$, we denote by $\|\cdot\|_{L^p(\partial T, \nu)}$ and $\|\cdot\|_{L^p(T, m)}$ the corresponding $L^p$-norms. Sometimes, the measure will be omitted from the norm subscript if it is clear from the context.
 
 \begin{definition}\label{def:Phi}
    We define the map $\Phi : \partial T \times \mathbb Z \to T$ such that $\Phi(\omega,j)$ is the unique vertex which lies above $\omega$ at level $j$. Namely, $\Phi$ is uniquely defined by $\Phi(\omega,j) \in (\omega,\omega_*)$ and $\ell(\Phi(\omega,j))=j$.
\end{definition}
Throughout the paper, we will make instrumental use of the Hardy--Littlewood maximal operator, whose definition we briefly recall.
Given a positive measure $\nu$ on $\partial T$ we define by $\cM$ the associated Hardy--Littlewood maximal operator, namely, 
\begin{align}
    \cM f(\omega_0)&=\sup_{r>0}\frac{1}{\nu(B_\rho(\omega_0,r))}\int_{B_\rho(\omega_0,r)} |f(\omega)|\ \mathrm{d}\nu(\omega)\nonumber \\
    &=\sup_{x \in (\omega_0,\omega_*)}\frac{1}{m_\nu(x)}\int_{\partial T_x}|f(\omega)| \ \mathrm{d}\nu(\omega) \qquad \forall \omega_0 \in \partial T, \label{HL:DEF}
\end{align} 
where $B_\rho(\omega_0,r)$ denotes the ball in $(\partial T,\rho)$ centered at $\omega_0$ with radius $r$ with respect to the distance $\rho$ and $f$ is a locally integrable function. \\ 
The next proposition establishes the weak type $(1,1)$ of the Hardy-Littlewood maximal operator on $(\partial T, \rho, \nu)$ independently of the choice of the measure $\nu$. This is a straightforward consequence of the covering properties of $\{{\partial T_x}\}_{x \in T}$ discussed above. For the sake of completeness, we provide a proof.
\begin{proposition}\label{p:HL}
    Let $\nu$ be a measure on $\partial T$ and $\cM$ denote the associated Hardy--Littlewood maximal operator on $\partial T$. Then,
    \begin{align*}
        \|\cM\|_{L^1(\partial T,\nu)\to L^{1,\infty}(\partial T,\nu)}\le 1
    \end{align*} and  $\cM$ is bounded on $L^p(\partial T,\nu)$ for every $p>1$.
\end{proposition} 
\begin{proof}
   It suffices to prove the weak type (1,1) boundedness since the statement about the $L^p$ bounds follows directly by interpolating with the obvious $L^\infty$ boundedness. Fix $\lambda>0$ and a function $f \in L^1(\partial T,\nu)$.  Define $E_\lambda=\{\cM f>\lambda\}$ and choose a collection of balls $\{\partial T_{x_j}\}_{j \in J}$ in $\partial T$ such that 
   \begin{align*}
       &E_\lambda \subset \cup_{j \in J} \partial T_{x_j}, \\ 
       &\frac{1}{\nu(\partial T_{x_j})}\int_{\partial T_{x_j}}|f| \ d\nu>\lambda \qquad \forall j \in J.
   \end{align*} Since two balls with nonempty intersection are such that one contains the other, we can extract a subcollection of pairwise disjoint balls $\{\partial T_{x_j}\}_{j \in J'}$ such that $$\cup_{j \in J} \partial T_{x_j}=\cup_{j \in J'} \partial T_{x_j}.$$ It follows that 
   \begin{align*}
       \nu(E_\lambda)\le \sum_{j \in J'} \nu(\partial T_{x_j}) \le \frac{1}{\lambda}\|f\|_{L^1(\partial T,\nu)}.
   \end{align*}

\end{proof} 
\section{Poisson integral and Carleson measures}\label{sec:2}
From now on we shall assume that $T$ is a locally finite tree, that is, every vertex in $T$ has a finite number of neighbours. We stress that we are not assuming that the number of neighbours of a vertex is a bounded function. While the local finiteness is not essential for part of our result, it simplifies our approach by avoiding some technicalities. \\ 
We fix once and for all a positive measure $\nu$ on $\partial T.$ We introduce a Laplace operator $\Delta$ on $T$ that will be central later. For a given function $f$ on $T$ we define 
\begin{align*}
    \Delta f(x)=f(x)-\sum_{y \in s(x)}f(y)\frac{m_\nu(y)}{m_\nu(x)}, \qquad \forall x \in T,
\end{align*}
where $m_\nu$ is the flow measure induced by $\nu$. \\ $\Delta$ is a {\it probabilistic Laplacian} in the sense that $$\Delta=I-P$$ and $P$ acts on a function $f$ by $Pf(x)=\sum_{y \in T}p(x,y)f(y)$ where
\begin{align*}
    0\le p(x,y)=\begin{cases} 0 &\text{if $x \not \sim y$ or $y=p(x)$},\\ 
    \frac{m_\nu(y)}{m_\nu(x)} &\text{if $y \in s(x)$},
    \end{cases} 
\end{align*} and 
\begin{align*}
    \sum_{y \in T}p(x,y)=1 \qquad \forall x \in T.
\end{align*}Notice that, since there are no restrictions on the flow measure $m_\nu$, $p(x,y)$ can be arbitrarily close to either $0$  or $1$ when $y \in s(x)$ and it is zero when $y=p(x)\sim x$. \\ 
We say that a function $f$ is {\it harmonic} on $S \subset T$ if $$\Delta f(x)=0\qquad \forall x \in S.$$

Observe that a harmonic function $f$ on $T$ satisfies
\begin{align}\label{harmonic}
    f(x)m_\nu(x)=\sum_{y \in s_n(x)}f(y)m_\nu(y) \qquad \forall x \in T, n \in \mathbb N.
\end{align} Indeed, \eqref{harmonic} holds trivially when $n=1$. 
We proceed by induction: assuming \eqref{harmonic} for some $n\ge 1$
and exploiting the fact that $f$ is harmonic for all $y\in s_n(x)$,
\begin{align*}
      f(x)m_\nu(x)&=\sum_{y \in s_n(x)}m_\nu(y)f(y)\\ &=\sum_{y \in s_n(x)}m_\nu(y)\sum_{z \in s(y)}\frac{m_\nu(z)}{m_\nu(y)}f(z)\\&=\sum_{z \in s_{n+1}(x)}m_\nu(z)f(z) \qquad \forall x \in T,
\end{align*} which proves \eqref{harmonic}.

Therefore, the limit
\begin{align}\label{poisson:sum}
    \lim_{n \to \infty} \sum_{y \in s_n(x)}f(y)\frac{m_\nu(y)}{m_\nu(x)}
\end{align} 
makes sense and it is equal to $f(x)$. Thus, since $s_n(x)$ tends to $\partial T_x$ as $n\to \infty$ in a suitable sense, \eqref{poisson:sum} suggests that if $f$ is good enough, then a Poisson integral representation formula for harmonic function holds. To this purpose, we need the following notion of continuous extension of a function on $\overline{T}$. 
\begin{definition}
    We say that a function $f$ on $T$ admits a continuous extension on $\overline{T}$ if there exists a function $g$ on $\partial T$ such that $\lim_{x \to \omega}f(x)=g(\omega)$  for a.e. $\omega \in \partial T$. The above limit is defined as follows: for every $\varepsilon>0$ and a.e. $\omega \in \partial T$ there exists a $\delta=\delta(\varepsilon,\omega)>0$ such that
    \begin{align*}
        \rho(\omega,x)<\delta \implies |f(x)-g(\omega)|<\varepsilon.
    \end{align*} With a slight abuse of notation we  denote by $f$ the continuous extension of $f$ to $\overline{T}$.
\end{definition} 
The following result provides a Poisson representation formula for harmonic functions. While the proof is not hard, we present the details for the reader's convenience.
\begin{proposition}
    Let $f$ be a bounded harmonic function on $T$ that admits a continuous extension to $\overline{T}$. Then, 
    \begin{align*}
    f(x)=\frac{1}{m_\nu(x)}\int_{\partial T_x} f(\omega) \ \mathrm{d}\nu(\omega) \qquad \forall x \in T.
\end{align*} Conversely,  for every $g \in L^1_{\loc}(\partial T, \nu)$, the function $f$ defined by
\begin{align*}
    f(x)=\frac{1}{m_\nu(x)}\int_{\partial T_x} g(\omega) \ \mathrm{d}\nu(\omega), \qquad \forall x \in T,
\end{align*} is harmonic on $T$.
\end{proposition}
\begin{proof}
 We first observe that $\|f\|_{L^\infty{(\partial T,\nu)}} \le \|f\|_{L^\infty(T,m_\nu)}.$ Indeed, for a.e. $\omega \in \partial T$ and $\varepsilon>0$ there is a $x_\varepsilon \in (\omega,\omega_*)$ such that
\begin{align*}
    |f(\omega)|\le |f(x_\varepsilon)|+\varepsilon \le \|f\|_{L^\infty(T,m_\nu)}+\varepsilon.
\end{align*} This implies that for every $x \in T$, $f$ is integrable on $\partial T_x$  and  thus
\begin{align}\label{intlebesgue}
    \int_{\partial T_x} f(\omega) \ \mathrm{d}\nu(\omega)
\end{align} makes sense. Moreover, by \eqref{harmonic}, for every $n \in \mathbb N$ 
\begin{align}\label{riemsum}
    m_\nu(x)f(x)= \sum_{y \in s_n(x)}f(y)m_\nu(y)=\sum_{y \in s_n(x)}f(y)\nu(\partial T_y).  
\end{align} Define the sequence of simple functions $f_{n,x}(\omega)=f_n(\omega)=\sum_{y \in s_n(x)}f(y)\chi_{\partial T_y}(\omega)$. It is clear that 
\begin{align}\label{numero1}
    \int_{\partial T}f_n(\omega)  \ \mathrm{d}\nu(\omega)=\sum_{y \in s_n(x)}f(y)\nu(\partial T_y)=f(x)m_\nu(x) \qquad\forall n \in \mathbb N,
\end{align} by \eqref{riemsum}. We notice that $\{\partial T_y\}_{y \in s_{n}(x)}$ is a partition of $\partial T_x$, hence  for every $\omega \in \partial T_x$ we have that $f_n(\omega)=f(\omega_n)$ where $\omega_n$ is the unique vertex in $s_n(x)$ such that $\omega \le \omega_n$. Moreover, $$\rho(\omega_n,\omega)=e^{\ell(\omega_n)}=e^{\ell(x)-n}\rightarrow 0 \qquad \hbox{as $n$ tends to $\infty$}.$$
Since $f$ admits a continuous extension, we deduce that 
\begin{align*}
    \lim_{n \to \infty}f_n(\omega)=f(\omega) \qquad a.e. \ \omega \in \partial T_x.
\end{align*}
Observe that $|f_n(\omega)|\le \|f\|_\infty \chi_{\partial T_x}(\omega)$, 
so we conclude by the Lebesgue Dominated Convergence Theorem that
\begin{align*}
    f(x)m_\nu(x)=\lim_{n \to \infty} \int_{\partial T_x}f_n(\omega)=\int_{\partial T_x} f(\omega) \ \mathrm{d}\nu(\omega).
\end{align*}
It is easy to see that the converse  holds: if $f$ is such that 
\begin{align*}
    f(x)=\frac{1}{m_\nu(x)}\int_{\partial T_x}g(\omega) \ \mathrm{d}\nu(\omega),
\end{align*} for some $g \in L^1_{\mathrm{loc}}(\partial T, \nu)$,  then
\begin{align*}
    \Delta f(x)=\int_{\partial T} \Delta K(x,\omega) g(\omega) \ \mathrm{d}\nu(\omega)=0, 
\end{align*} where $K(x,\omega):=\frac{\chi_{\partial T_x}(\omega)}{m_\nu(x)}$. Indeed, it is readily seen that for every $x \in T$ and $\omega \in \partial T$ 
\begin{align*}
    \Delta K(\cdot,\omega)(x)&=\frac{1}{m_\nu(x)}\chi_{\partial T_x}(\omega)-\sum_{y \in s(x)}\frac{m_\nu(y) \chi_{\partial T_y}(\omega)}{m_\nu(x)m_\nu(y)} \\ 
    &=\begin{cases} 0
        &\text{if $\omega \not \in \partial T_x$}, \\ 
        \frac{1}{m_\nu(x)}-\frac{1}{m_\nu(x)}\times 1 &\text{if $\omega \in \partial T_x$} 
    \end{cases} \\ 
    &=0.
\end{align*}
\end{proof} For the remaining of the paper, we write $\cP$ for the integral operator defined by 
\begin{align}\label{def:K}
    \cP f(x)=\int_{\partial T}P(x,\omega) f(\omega) \ \mathrm{d}\nu(\omega),
\end{align} where $P(x,\omega)=\displaystyle \frac{\chi_{\partial T_x}(\omega)}{m_\nu(x)}$. We will refer to $\cP$ as the {\it Poisson integral operator}. \\ 
In Theorem \ref{th: charactPoisson}, we will characterise harmonic functions that are the Poisson integral of a function in $L^p(\partial T,\nu)$. For a different approach on the Poisson representation of 
 harmonic functions on trees, we refer the reader to \cite{PW} and the references therein.
In order to state our result, we need to introduce a suitable Hardy space.
\begin{definition} For every $p \ge 1$,
    we say that a harmonic function $f$ on $T$ belongs to $H^p$ if 
    \begin{align*}
\|f\|_{H^p}^p&:=\sup_{k \in \mathbb Z} \sum_{\ell(x)=k}| f(x)|^pm_\nu(x)<\infty \!\!\!\!\!&&\hbox{if $p<\infty$}, \\ 
\|f\|_{H^\infty}&:=\|f\|_{L^\infty(T)}<\infty &&\hbox{if $p=\infty$}.
    \end{align*} 
\end{definition}
$H^p$ can be thought of as the analogue of the Hardy spaces on the upper half-plane (see for example \cite[Chapter 2]{Garnett}). \\ 
 Observe that if $g \in L^p(\partial T,\nu)$ then by Jensen's inequality
\begin{align}\label{necessary}
    \|\cP g\|_{H^p}^p=\sup_{k \in \mathbb Z} \sum_{\ell(x)=k}\bigg|\frac{1}{m_\nu(x)}\int_{\partial T_x} g(\omega) \ \mathrm{d}\nu(\omega)\bigg|^pm_\nu(x) \le \|g\|_{L^p(\partial T)}^p,
\end{align} because $\{\partial T_x\}_{\ell(x)=k}$ is a partition of $\partial T$.
\begin{theorem}\label{th: charactPoisson}Let $f$ be a harmonic function on $T$ and $p>1$. Then, $f$ is the Poisson integral of a $L^p(\partial T,\nu)$ function if and only if $f \in H^p$.
\end{theorem}
\begin{proof}
    The necessary condition follows by \eqref{necessary}. For the other direction, assume that $f \in H^p$ for some $p>1$. 
    We provide the details of the proof for $p<\infty$. The proof for $p=\infty$ is analogous with obvious modifications.
    We aim to show that $f=\cP g$ for a suitable $g \in L^p(\partial T,\nu)$.
   We first observe that it suffices to prove that there exists $g \in L^p(\partial T,\nu)$ such that  \begin{align}\label{asterisco}
    \cP g(x)=f(x) \qquad \forall x \ : \ \ell(x) \le 0.  
   \end{align} Indeed, assume that \eqref{asterisco} holds. If $z$ has level $>0$ then by \eqref{harmonic} we see that  
   \begin{align}
     \nonumber  f(z)&=\frac{1}{m_\nu(z)}\sum_{y \in s_{\ell(z)}(z)}f(y)m_\nu(y)\\ \nonumber &=\frac{1}{m_\nu(z)}\sum_{y \in s_{\ell(z)}(z)}\frac{m_\nu(y)}{m_\nu(y)}\int_{\partial T_y} g(\omega) \ \mathrm{d}\nu(\omega)\\ \label{asterisco2} &=\int_{T} P(z,\omega) g(\omega) \ \mathrm{d}\nu(\omega),
   \end{align} which concludes the proof. \\ 
  We prove \eqref{asterisco}. Let $x\in T$ be such that $\ell(x)=0$. For every $n  \in \mathbb N$, set $$T^n_x=T_x \cap \{y \ : \ \ell(y) \ge -n\}$$ and consider the sequence of functions $u_x^n$ defined by 
    \begin{align*}
u_x^n(y)=\begin{cases}
            f(y) &\text{if $y \in T_x^n$,}\\ 
            f(x) &\text{if $y \not \in T_x$,} \\ 
            f(p^{-n-\ell(y)}(y)) &\text{if $y \in T_x \setminus T_x^n$.}
        \end{cases} 
        \end{align*}
        We remark that if $y \in T_x \setminus T^n_x$, $p^{-n-\ell(y)}(y)$ is the closest vertex in $T^n_x$ to $y$. Note that $u^n_x$ coincides with $f$ on $T_x^n$ and for every $y \in ({T_x^n})^c$ one has $u^n_x(y)= u^n_x(z)$ for every $z \sim y.$ 
        We deduce that $u_x^n$ is harmonic on $T$,  constant on $T_y$ for every $y$ such that $\ell(y)\le -n$, and $$\sup_{y \in T}|u_x^n(y)| \le \max_{z \in T^n_x} |f(z)|.$$ 
        Thus, by Proposition \ref{poisson:sum} we conclude that for every $y \in T$  $$u_x^n(y)=\int_{\partial T}P(y,\omega) u_x^n(\omega) \ \mathrm{d}\nu(\omega),$$ where $u_x^n$ also denotes  the continuous extension of $u_x^n$ on $\partial T$, which exists because $u_x^n$ is constant on $T_y$ for every $y$ with level $\le-n$. In particular, for every $y \in T_x$ and $n \ge -\ell(y)$, namely, for every $n$ such that $y \in T_x^n$, we have that
\begin{align}\label{identityf}f(y)=u_x^n(y)=\int_{\partial T} P(y,\omega) u_x^n(\omega) \  \mathrm{d}\nu(\omega)=\int_{\partial T_x} P(y,\omega) u_x^n(\omega) \  \mathrm{d}\nu(\omega).\end{align}
        Observe that the sequence $\{u^n_x\}_{n}$ is bounded on $L^p(\partial T_x,\nu)$ because 
        \begin{align}\label{boundun}
            \int_{\partial T_x} |u^n_x(\omega)|^p \ \mathrm{d}\nu(\omega)=\sum_{z \in s_n(x)} |f(z)|^p m_\nu(z)\le \|f\|_{H^p}^p \qquad \forall n \in\mathbb N,
\end{align} where we have used that for every $z\in s_n(x)$, $u_x^n=f(z)$ on $\partial T_z$ and $\partial T_x=\cup_{z \in s_n(x)}\partial T_z$. Next, the Banach-Alaoglu Theorem implies that $\{u_x^n\}_n$ admits a subsequence $\{u^{n_k}_x\}_k$ that weakly converges in $L^p(\partial T_x, \nu)$ to a function $u_x$. It follows by \eqref{identityf} that for every $y \in T_x$
        \begin{align*}
            f(y)=\lim_{k\to \infty} u_x^{n_k}(y)=\lim_{k \to \infty}\int_{\partial T_x} P(y,\omega)u_x^{n_k}(\omega) \ \mathrm{d}\nu(\omega)=\int_{\partial T_x} P(y,\omega)u_x(\omega) \ \mathrm{d}\nu(\omega),
        \end{align*} because $P(y,\cdot) \in L^q(\partial T_x,\nu)$ for every $y \in T_x$, where $q=p/(p-1)$. We conclude that for every $y \in T_x$,  $f(y)=\cP u_x(y)$. \\ We  set $$g(\omega)=u_x(\omega) \qquad \forall \omega \in \partial T_x.$$ Then, by the arbitrariness of $x$ in $\{y \in T \ : \ \ell(y)=0\}$ and \eqref{asterisco2}, we conclude  $\cP g=f$ on $T$.  \\ It remains to prove that $g \in L^p(\partial T, \nu)$. We observe that for every $n \in \mathbb N$ and $x \in T$ with $\ell(x)=0$ $$\int_{\partial T_x}|u^n_x(\omega)|^p \ \mathrm{d}\nu(\omega)=\sum_{z \in s_n(x)} |f(z)|^p m_\nu(z) \le \sum_{z \in s_{n+1}(x)} |f(z)|^p m_\nu(z),$$
        where the last inequality follows from the fact that $f$ is harmonic for every $z \in s_n(x)$ and an application of Jensen's inequality. It follows that $n \mapsto \int_{\partial T_x}|u^n_x(\omega)|^p \ \mathrm{d}\nu(\omega) $
        is increasing and bounded because $f \in H^p$. Then, on the one hand
        \begin{align}\label{limitexists}
            \lim_{n \to \infty}\int_{\partial T_x}|u^n_x(\omega)|^p \ \mathrm{d}\nu(\omega)=\lim_{n \to \infty}\sum_{z \in s_n(x)}|f(z)|^p m_\nu(z)=c \le \|f\|_{H^p}^p<\infty.
        \end{align} On the other hand it is known that the $p$-norm is weakly lower semicontinuous, thus 
        \begin{align}\label{wlsc}
        \int_{\partial T_x} |u_x(\omega)|^p \ \mathrm{d}\nu(\omega) \le     \liminf_{k \to \infty} \int_{\partial T_x}|u^{n_k}_x(\omega)|^p\ \mathrm{d}\nu(\omega)=\lim_{k \to \infty} \int_{\partial T_x}|u^{k}_x(\omega)|^p\ \mathrm{d}\nu(\omega),
        \end{align}
        where in the second equality we have used \eqref{limitexists}.
We conclude that
\begin{align*}
           \|g\|_{L^p(\partial T,\nu)}^p&=\sum_{x : \ell(x)=0} \int_{\partial T_x} |u_x(\omega)|^p \ \mathrm{d}\nu(\omega) \\&=\lim_{n \to \infty}\sum_{x : \ell(x)=0, d(x,o)<n} \int_{\partial T_x} |u_x(\omega)|^p \ \mathrm{d}\nu(\omega) 
           \\&\le \lim_{n \to \infty}\sum_{x : \ell(x)=0, d(x,o)<n} \lim_{k \to \infty}\int_{\partial T_x} |u_x^{k}(\omega)|^p \ \mathrm{d}\nu(\omega) \\&= \lim_{n \to \infty} \lim_{k \to \infty}\sum_{x : \ell(x)=0, d(x,o)<n} \sum_{z \in s_k(x)}|f(z)|^p m_\nu(z) \\ 
           &\le \lim_{n \to \infty}\lim_{k \to \infty}\sum_{x : \ell(x)=-k}  |f(z)|^p m_\nu(z) \\ 
           &\le \|f\|_{H^p}^p.
           \end{align*}
\end{proof} 
 We are now ready to define Carleson measures. 
\begin{definition}  We say that a positive measure $\sigma$ on $T$ is a Carleson measure if there exists a constant $C>0$ such that for every $x \in T$ $$\sigma(T_x)\le C \nu(\partial T_x).$$
\end{definition}  
Like in the Euclidean case, we show that Carleson measures are related to the boundedness of $\cP$ between suitable $L^p$ spaces. To accomplish this, we first discuss some properties of the Poisson integral operator. 
\\ We define the maximal operator $\mathcal{U}$ by
\begin{align*}
\mathcal{U}f(\omega)=\sup_{x \in (\omega,\omega_*)}|f(x)|, \qquad \forall \omega \in \partial T,
\end{align*} 
 where $f$ is a function on $T$.
\begin{proposition}Let $\Phi$ be as in Definition \ref{def:Phi}. The following hold:
\begin{itemize}
\item[$i)$] for all $f\in L^1_{\mathrm{loc}}(\partial T,\nu)$ \begin{align}\label{f:UK}
\cU \cP f (\omega)\le \cM f(\omega) \qquad \forall \omega \in \partial T,
\end{align} 
 and thus in particular
\begin{align}
\cP f(\Phi(\omega,j)) \le \cM f(\omega) \qquad \forall j \in \mathbb Z,  \omega \in \partial T; \end{align} \item[$ii)$] for a.e. $\omega \in \partial T$, $p\in (1,\infty)$ and $f\in L^p(\partial T, \nu)$  
\begin{align*}
   \lim_{j \to -\infty} \cP f(\Phi(\omega,j))=f(\omega)
\end{align*} 
and
\begin{align}\label{f:DCT}
    \lim_{j \to -\infty} \int_{\partial T} |\cP f(\Phi(\omega,j))-f(\omega)|^p \ d\nu(\omega)=0;
\end{align}
\item[$iii)$] for every $x \in T$ 
\begin{align}\label{intK=1}
    \int_{\partial T}P(x,\omega) \ \mathrm{d}\nu(\omega)=1.
\end{align}
\end{itemize}
\end{proposition}
\begin{proof}
    Let $f \in L^1_{\mathrm{loc}}(\partial T,\nu)$. Then, by \eqref{HL:DEF}
    \begin{align*}
        \cU \cP f (\omega)\le \sup_{x \in (\omega,\omega_*)}\frac{1}{m_\nu(x)}\int_{\partial T_x}|f(\omega)| \ \mathrm{d}\nu(\omega)= \cM f(\omega) \qquad\forall \omega \in \partial T,
    \end{align*} which is $i)$. In order to prove $ii)$, assume that $f \in L^p(\partial T,\nu)$ for some $p \in (1,\infty)$. Since $\partial T$ is a locally compact space on which $\cM$ is of weak type $(1,1)$, the Lebesgue differentiation Theorem holds. Thus for a.e. $\omega_0 \in \partial T$,
\begin{align}\label{diffleb}
   \lim_{j \to -\infty} \cP f(\Phi(\omega_0,j))=\lim_{j \to - \infty} \frac{1}{\nu(\partial T_{\Phi(\omega_0,j)})}\int_{\partial T_{\Phi(\omega_0,j)}} f(\omega) \ d\nu(\omega)=f(\omega_0).
\end{align}  By $i)$, the fact that $\cM f \in L^p(\partial T,\nu)$, \eqref{diffleb} and the Dominated Converge Theorem, we have that \eqref{f:DCT} holds. Finally, $iii)$ follows from a straightforward computation.
\end{proof} We remark that since $P(x,\omega) \ge 0$ for every $(x,\omega) \in T \times \partial T$, $iii)$ in the above proposition implies that for every positive  measure $\sigma$ on $T$,
\begin{align}\label{remarkK}
    \|\cP f\|_{L^\infty(T,\sigma)} \le  \|f\|_{L^\infty(\partial T,\nu)}, \qquad \forall f \in L^\infty(\partial T,\nu).
\end{align}
\begin{theorem}\label{th: Poisson}

Let $\nu$ be a positive measure on $\partial T$. The following facts are equivalent  \begin{itemize}
    \item[$i)$] $\sigma$ is a Carleson measure; 
      \item[$ii)$] there exists a $C>0$ such that for every $p>1$ and  $f \in L^p(\partial T,\nu)$ $$\|\cP f\|_{L^p(T,\sigma)}\le C\| f\|_{L^p(\partial T,\nu)},$$ and $$\|\cP f\|_{L^{1,\infty}(T,\sigma)}\le C \|f\|_{L^1(\partial T,\nu)};$$
    \item[$iii)$] there exists  $C>0$ such that for every $p>1$ and $f \in L^p(\partial T,\nu)$ $$\|\cP f\|_{L^p(T,\sigma)}\le C\|\cP f\|_{H^p}.$$
\end{itemize}   
\end{theorem}
\begin{proof}
    Assume that $\sigma$ is a Carleson measure. To prove $ii)$, by \eqref{remarkK} it suffices to prove the weak type $(1,1)$ boundedness of $\cP$ and interpolate. For $\lambda>0$ and $f \in L^1(\partial T, \nu)$, set $F_\lambda=\{x \in T  \ : |\cP f(x)|>\lambda\}$. 
  The fact that $x \in T_x$ for every $x \in T$ readily implies that 
\begin{align*}
    F_\lambda \subset \cup_{x \in F_\lambda} T_x.
\end{align*} Since when $T_x \cap T_y \ne \emptyset$ we have that $T_x \subset T_y$ or $T_y \subset T_x$, there exists a nonempty set $F'_\lambda\subset F_\lambda$ such that for every $x,y \in F'_\lambda$ we have that $T_x \cap T_y\ne \emptyset$ implies that $x=y$ and $$\cup_{x \in F'_\lambda} T_x=\cup_{x \in F_\lambda} T_x.$$ 
It follows that 
\begin{align*}
    F_\lambda \subset \cup_{x \in F'_\lambda} T_x,
\end{align*} where $\{T_x\}_{x \in F'_\lambda}$ are pairwise disjoint. We conclude that 
\begin{align}\label{disjoint-boundary}
    \sigma(F_\lambda) \le \sum_{x \in F'_\lambda} \sigma(T_x) \le C \sum_{x \in F'_\lambda} \nu(\partial T_x).
\end{align} Next, observe that if $\cP f(x)>\lambda$ then $\cU \cP f(\omega)>\lambda$ for every $\omega \in \partial T_x$. This means that $\partial T_x \subset \{\cU (\cP f)>\lambda\}$ for every $x \in F'_\lambda$. Observing that $\{\partial T_x\}_{x \in F'_\lambda}$ are pairwise disjoint,  \eqref{disjoint-boundary} implies
\begin{align*}
    \sigma(F_\lambda) \le C \sum_{x \in F'_\lambda}\nu(\partial T_x \cap \{\cU (\cP f)>\lambda\})  \le C \nu(\{\cU (\cP f)>\lambda\}).
\end{align*}  We conclude by \eqref{f:UK} that     \begin{align*}
    \sigma(F_\lambda) \le C \nu(\{\cM f>\lambda\}),
\end{align*} and now the result follows by Proposition \ref{p:HL}.  
  \\ Next, we  show that $ii)$ implies $iii)$. Indeed, assume that for $p>1$
  \begin{align}\label{4}
      \|\cP f\|_{L^p(T,\sigma)} \le C\|f\|_{L^p(\partial T,\nu)} \qquad \forall f \in L^p(\partial T,\nu).
  \end{align}
Then  by \eqref{f:DCT} and Fatou's Lemma, for every $f \in L^p(\partial T,\nu)$ and $k \in \mathbb Z$
\begin{align}\nonumber
\int_{\partial T}|f(\omega)|^p \ \mathrm{d}\nu(\omega)&=\sum_{x \ : \ell(x)=k}\int_{\partial T_x} |f(\omega)|^p \ \mathrm{d}\nu(\omega) \\ \nonumber &=  \sum_{x \ : \ell(x)=k} \lim_{j \to -\infty}\|\cP f (\Phi(\cdot,j))\|^p_{L^p(\partial T_x)} \\ \nonumber 
&\le \liminf_{j \to -\infty}\sum_{x \ : \ell(x)=k} \int_{\partial T_x} |\cP f(\Phi(\omega,j))|^p \ \mathrm{d}\nu(\omega) \\ \nonumber
&\le \sup_{j \le k} \sum_{x \ : \ell(x)=k} \int_{\partial T_x} |\cP f(\Phi(\omega,j))|^p \ \mathrm{d}\nu(\omega) \\ \label{(numero)}
&=\sup_{j \le k} \sum_{x \ : \ell(x)=k} \sum_{\substack{y \le x \\ \ell(y)=j}} |\cP f(y)|^p m_\nu(y) \\\nonumber 
&=\sup_{j \le k} \sum_{y : \ell(y)=j} |\cP f(y)|^p m_\nu(y) \\\label{3}&\le \|\cP f\|_{H^p}^p, 
    \end{align} where in \eqref{(numero)} we have used that $\omega \mapsto \cP f (\Phi(\omega,j))$ is constant on $\partial T_{\Phi(\omega,j)}$ for every $j$ fixed. Next, \eqref{4}  together with \eqref{3} imply $iii)$. \\  Now assume $iii)$    and consider $f=\chi_{\partial T_v}$ for some $v \in T$. By \eqref{intK=1}, on the one hand, we have  that
    \begin{align}\nonumber
        \|\cP f\|_{L^p(T,\sigma)}^p&=\sum_{x \in T} \bigg|\int_{\partial T}P(x,\omega) f(\omega) \ \mathrm{d}\nu(\omega) \bigg|^p \sigma(x)\\ \nonumber &\ge\sum_{x \le v} \bigg|\int_{\partial T}P(x,\omega) f(\omega) \ \mathrm{d}\nu(\omega) \bigg|^p \sigma(x)\\ \nonumber
        &=\sum_{x \le v} \bigg|\frac{1}{m_\nu(x)} \int_{\partial T_v} \chi_{\partial T_x}(\omega) \ \mathrm{d}\nu(\omega)\bigg|^p \sigma(x)\\\nonumber 
        &=\sum_{x \le v} \sigma(x)\\ \label{2}
        &=\sigma(T_v).
    \end{align}
  In the second last equality, we have used that $\partial T_x \subset \partial T_v$ for every $x \le v$ and $\nu(\partial T_x)=m_\nu(x)$.   On the other hand, we claim that
    \begin{align}\label{1}
        \|\cP f\|_{H^p}^p=m_\nu(v).
\end{align} This would conclude the proof since $i)$ follows by combining \eqref{1},\eqref{2}, and $iii)$. \\  To prove \eqref{1}, we observe that 
$$\cP f(x)=\frac{1}{m_\nu(x)} \nu(\partial T_x \cap \partial T_v),$$
which is non-zero if and only if $x \le v$ or $v \le x$. In the first cases $\cP f(x)=1$ and in the second case $\cP f(x)=\displaystyle\frac{m_\nu(v)}{m_\nu(x)}$.

Thus, for every $k \le \ell(v)$ the above considerations and the definition of flow measure yield
\begin{align*}
    \sum_{\ell(x)=k}|\cP f(x)|^p m_\nu(x)=\sum_{x \le v, \ell(x)=k}  m_\nu(x)=m_\nu(v),
\end{align*} while for every $k>\ell(v)$, the unique vertex with level $k$ that lies above $v$ is $p^{k-\ell(v)}(v)$ and
\begin{align*}
  \sum_{\ell(x)=k}|\cP f(x)|^p m_\nu(x) &=  \bigg(\frac{m_\nu(v)}{m_\nu(p^{k-\ell(v)}(v))}\bigg)^pm_\nu(p^{k-\ell(v)}(v)) \\ &=\frac{m_\nu(v)^p}{m_\nu(p^{k-\ell(v)}(v))^{p-1}}.
\end{align*}
We conclude by observing that 
$$\frac{m_\nu(v)^p}{m_\nu(p^{k-\ell(v)}(v))^{p-1}} \le m_\nu(v)$$ since the above is equivalent to 
\begin{align*}
    m_\nu(v) \le m_\nu(p^{k-\ell(v)}(v)),
\end{align*} which is always true for $k \ge \ell(v)$ since $m_\nu$ is a flow measure. This proves \eqref{1} and concludes the proof.
  \end{proof}

\section{Carleson measures and BMO}\label{sec:3} 
In this section, we assume that the number of successors of every vertex is bigger than or equal to two. 
In \cite[Proposition 2.2]{LSTV} it is proved that $m$ is a locally doubling flow measure on $T$ if and only if there are two positive constants $c_1,c_2>1$  such that
\begin{align}\label{constants-locdoub}
    c_2 m(y)  \le m(x) \le c_1 m(y), \qquad \forall x \in T, y \in s(x). 
\end{align}
In the same proposition it is also proved that the inequality $m(x) \le c_1 m(y)$ for every $x \in T$ and $y \in s(x)$ implies  $m(x) \ge c_1/(c_1-1) m(y)$ for every $x \in T$ and $y \in s(x)$. 
Moreover, it is known that if $m$ is locally doubling, then the number of neighbours of a vertex is bounded on $T$, see \cite[Corollary 2.3]{LSTV}.
\begin{lemma} Let $\nu$ be a positive measure on $\partial T$. Then, $(\partial T,\rho,\nu)$ is doubling if and only if $(T,d,m_\nu)$ is locally doubling.
\end{lemma}
\begin{proof} Fix $\omega_0 \in \partial T$ and observe that for a given $r>0$ 
\begin{align*}
    B_\rho(\omega_0,r)=\{ \omega \in \partial T \ : \ \ell(\omega \wedge \omega_0) \le \log r\}=\partial T_{\Psi(\omega_0,r)},
\end{align*} where $\Psi(\omega_0,r)$ is the unique vertex such that $\ell(\Psi(\omega_0,r))=\lfloor \log r\rfloor$ and $\Psi(\omega_0,r) \in (\omega_0,\omega_*)$. Then, $\nu(B_\rho(\omega,r))=m_\nu(\Psi(\omega_0,r))$. Similarly, 
\begin{align*}
B_\rho(\omega_0,2r)=\partial T_{\Psi(\omega_0,2r)},
\end{align*} and $\ell(\Psi(\omega_0,2r))=\lfloor \log 2r\rfloor \le \lfloor \log r \rfloor+1=\ell(\Psi(\omega_0,r))+1=\ell(p(\Psi(\omega_0,r))).$ Thus $(\partial T, \rho, \nu)$ is doubling if and only if there exists a constant $C>0$ such that for every $\omega_0 \in \partial T$ and $r>0$
\begin{align}\label{equivalence}
\frac{m_\nu(\Psi(\omega_0,2r))}{m_\nu(\Psi(\omega_0,r))} \le C.
\end{align} For every $x \in T$ and $\omega \in \partial T_x$,  we choose $r=e^{\ell(x)+1-\log 2}$ and we get that $\Psi(\omega,r)=x$ and $\Psi(\omega,2r)=p(x)$. Therefore, we deduce that \eqref{equivalence} is equivalent to $$\frac{m_\nu(p(x))}{m_\nu(x)}\le C \qquad \forall x \in T,$$ so we conclude by invoking \eqref{constants-locdoub} and the discussion thereafter.
\end{proof} From now on we shall assume that $\nu$ is a doubling measure on $\partial T$. This allows us to consider on $\partial T$ the $BMO$ space defined in \cite{CW2}. More precisely, we set $$BMO=\{b\in L^1_{\mathrm{loc}}(\partial T,\nu) \ : \ \|b\|_{BMO}<\infty\},$$ where $$\|b\|_{BMO}:=\sup_{x \in T}\frac{1}{\nu(\partial T_x)}\int_{\partial T_x}|b(\omega)-b_{\partial T_x}| \ \mathrm{d}\nu(\omega)$$ and for every $E \subset \partial T$ we set  $b_E=\frac{1}{\nu(E)}\int_{E} b \ \mathrm{d}\nu$.
\begin{definition}\label{def:O}
    We define a space of integral operators that we denote by $\cO$ such that every $\cK$ in $\cO$ acts on functions on $\partial T$ and it has an integral kernel $K(\cdot,\cdot) : T \times \partial T \to \mathbb C$ that satisfies the following cancellation, integrability and decay properties:
\begin{enumerate}
        \item\label{A1} for every fixed $x_0 \in T$ the map $\partial T \ni \omega \mapsto K(x_0,\omega)$ is integrable and $$\int_{\partial T}K(x_0,\omega)  \ \mathrm{d}\nu(\omega)=0;$$  
        \item\label{A2}  $$C_K:=\esssup_{\omega \in \partial T}\ \sum_{x \in T}|K(x,\omega)|m_\nu(x)<\infty;$$ 
        \item\label{A3} there exists $\alpha>0$ such that $|K(x,\omega)|\le \frac{m_\nu(x)^{\alpha}}{m_\nu(x\wedge \omega)^{\alpha+1}}$, for every $x \in T$ and $\omega \in \partial T$. 
\end{enumerate} 
\end{definition}
\begin{theorem}\label{th: BMO}
   Let $b$ be a locally integrable function. Then, the following are equivalent facts: 
    \begin{itemize}
    \item[$i)$] $b \in BMO$;
        \item[$ii)$] there exists a function $f:[0,\infty)\times (0,\infty) \to [0,\infty)$ depending on $b$ that is increasing in the first variable and for every $\cK \in \cO$ the measure $\sigma$ defined by $\sigma=|\cK b|m_\nu$ is a Carleson measure satisfying $$\sigma(T_v)\le f(C_K,\alpha)m_\nu(v) \qquad \forall v \in T,$$ where $C_K$ is an in \eqref{A2} and $\alpha$  as in \eqref{A3}.
    \end{itemize}
\end{theorem}
\begin{proof} We first prove that $i)$ implies $ii)$. \\ 
    Fix $v \in T$ and assume that $b \in BMO$. We have that
    \begin{align*}
        \sigma(T_v)=\sum_{x \in T_v} |\cK b(x)| m_\nu(x).
    \end{align*}  Observe that, by $\eqref{A1}$, 
    $$|\cK b(x)|\le |\cK[(b-b_{\partial T_{p(v)}})\chi_{\partial T_{p(v)}}](x)|+|\cK[(b-b_{\partial T_{p(v)}})\chi_{\partial T_{p(v)}^c}](x)|.$$ For notational convenience, we set $g_v=b-b_{\partial T_{p(v)}}$. Thus,
    \begin{align*}
        \sigma(T_v) \le \sum_{x \in T_v} |\cK (g_v\chi_{\partial T_{p(v)}})(x)|m_\nu(x)+\sum_{x \in T_v} |\cK (g_v\chi_{\partial T_{p(v)}^c})(x)|m_\nu(x)=:I_1+I_2.
    \end{align*} Observe that by  $\eqref{A2}$
    \begin{align*}
        I_1&\le \sum_{x \in T_v} \int_{\partial T_{p(v)}}|K(x,\omega)||g_v(\omega)| \ \mathrm{d}\nu(\omega) m_\nu(x) \\ 
        &=\int_{\partial T_{p(v)}}|g_v(\omega)| \sum_{x \in T_v}|K(x,\omega)|m_\nu(x) \ \mathrm{d}\nu(\omega) \\ 
        &\le c_1C_K \|b\|_{BMO} m_\nu(v),
    \end{align*} because $m_\nu(p(v))\le c_1 m_\nu(v)$ by \eqref{constants-locdoub}.  Similarly,  since $$\partial T_v^c=\cup_{k=0}^\infty \partial T_{p^{k+1}(v)} \setminus \partial T_{p^{k}(v)}$$ and $\{T_{p^{k+1}(v)} \setminus \partial T_{p^{k}(v)}\}_{k=0}^\infty$ are pairwise disjoint,  
    \begin{align*}
        I_2 &\le \sum_{x \in T_v} \sum_{k=0}^\infty \int_{\partial T_{p^{k+1}(v)}\setminus \partial T_{p^k(v)}}|K(x,\omega)| |g_v(\omega)| \ \mathrm{d}\nu(\omega) m_\nu(x)=:\sum_{x \in T_v} \sum_{k=0}^\infty J_k m_\nu(x).
    \end{align*}  Observe that  for every $x \in T_v$ and $\omega \in \partial T_{p^{k+1}(v)} \setminus \partial T_{p^k(v)}$ we have that $x\wedge \omega=p^{k+1}(v)$. Hence, by   $\eqref{A3}$
    \begin{align*}
        J_k &= \int_{\partial T_{p^{k+1}(v)}\setminus \partial T_{p^k(v)}}|K(x,\omega)| |g_v(\omega)| \ \mathrm{d}\nu(\omega) \\ &\le \int_{\partial T_{p^{k+1}(v)}\setminus \partial T_{p^k(v)}} \frac{m_\nu(x)^\alpha}{m_\nu(p^{k+1}(v))^{\alpha+1}}|g_v(\omega)| \ \mathrm{d}\nu(\omega) \\ 
        &\le \frac{m_\nu(x)^\alpha}{m_\nu(p^{k+1}(v))^{\alpha+1}}\int_{\partial T_{p^{k+1}(v)}} |g_v(\omega)| \ \mathrm{d}\nu(\omega)=: \frac{m_\nu(x)^\alpha}{m_\nu(p^{k+1}(v))^{\alpha+1}} S_k.
    \end{align*} Moreover, we set $v_k=p^k(v)$ for every $k \ge 0$ and observe that 
    \begin{align*}
        S_k&=\int_{\partial T_{v_{k+1}}} |b(\omega)-b_{\partial T_{v_1}}| \ \mathrm{d}\nu(\omega)\\&\le \int_{\partial T_{v_{k+1}}} |b(\omega)-b_{\partial T_{v_{k+1}}}| \ \mathrm{d}\nu(\omega)+ \int_{\partial T_{v_{k+1}}} |b_{\partial T_{v_{k+1}}}-b_{\partial T_{v_1}}| \ \mathrm{d}\nu(\omega) 
        \\&\le m_\nu(v_{k+1})\|b\|_{BMO}+m_\nu(v_{k+1})\sum_{j=1}^{k}|b_{\partial T_{v_{j+1}}}-b_{\partial T_{v_j}}|
        \\&\le c_1 m_\nu(v_{k+1})(k+1) \|b\|_{BMO}, 
    \end{align*} because
    \begin{align*}
        |b_{\partial T_{v_{j+1}}}-b_{\partial T_{v_{j}}}|&\le \frac{1}{m_\nu(v_j)}\int_{\partial T_{v_j}}|b(\omega)-b_{\partial T_{v_{j+1}}}| \ \mathrm{d}\nu(\omega) \\  &\le \frac{m_\nu(v_{j+1})}{m_\nu(v_{j})}\frac{1}{m_\nu(v_{j+1})}\int_{\partial T_{v_{j+1}}}|b(\omega)-b_{\partial T_{v_{j+1}}}| \ \mathrm{d}\nu(\omega) \\ &\le c_1\|b\|_{BMO}, \qquad \forall j \in \mathbb N,
    \end{align*} where the last inequality follows by \eqref{constants-locdoub}. 
 We conclude that
    \begin{align}\label{stimaI2}
        I_2\le c_1\|b\|_{BMO} \sum_{x \in T_v} m_\nu(x)^{\alpha+1}\sum_{k=0}^\infty (k+1){m_\nu(p^{k+1}(v))^{-\alpha}}.
    \end{align} Since ${\alpha}>0$ we claim that there exists a constant $C_\alpha>0$ such that \begin{align}
    \label{claim1}\sum_{k=0}^\infty (k+1) m_\nu(p^{k+1}(v))^{-\alpha} &\le C_\alpha m_\nu(v)^{-\alpha}
    \end{align} and
    \begin{align} 
    \label{claim2}\sum_{x \in T_v} m_\nu(x)^{1+\alpha} &\le C_\alpha m_\nu(v)^{1+\alpha}.
    \end{align} The above claim, together with \eqref{stimaI2}, implies that
    \begin{align*}
    I_2 \le C_\alpha \|b\|_{BMO} \sum_{x \in T_v} m_\nu(x)^{\alpha+1}{m_\nu(v)^{-\alpha}}\le  C_\alpha \|b\|_{BMO}m_\nu(v). 
    \end{align*} This, combined with the estimates involving $I_1$, yields
    $$\sigma(T_v) \le f(C_K,\alpha)m_\nu(v)$$
    where $f(C_K,\alpha):=\|b\|_{BMO}(c_1C_K+C_\alpha)$, concluding the first part of the proof. It remains to prove  \eqref{claim1} and \eqref{claim2}. By \eqref{constants-locdoub},
    \begin{align*}
        m_{\nu}(p^{k+1}(v)) \ge c_2^k m_\nu(v),
    \end{align*} thus 
    \begin{align*}
        \sum_{k=0}^\infty (k+1) m_\nu(p^{k+1}(v))^{-\alpha} \le \sum_{k=0}^\infty (k+1) c_2^{-k\alpha} m_{\nu}(v)^{-\alpha} \le C_\alpha m_\nu(v)^{-\alpha},
    \end{align*} because $c_2>1$. This proves \eqref{claim1}. Similarly, by \eqref{constants-locdoub} again, we have that $m_\nu(x) \le m_\nu(v) c_2^{-k}$ for every $x \in s_k(v)$ and $k \in \mathbb N$. Thus
        \begin{align}
        \nonumber\sum_{x \in T_v} m_\nu(x)^{\alpha+1}&=\sum_{k=0}^{\infty}\sum_{x \in s_k(v)}m_\nu(x)^{\alpha}m_\nu(x) \\\nonumber &\le \sum_{k=0}^\infty m_\nu(v)^{\alpha} c_2^{-\alpha k} \sum_{x \in s_k(v)}m_\nu(x) \\ &\label{flowmeasurecond}=m_\nu(v)^{\alpha+1} \sum_{k=0}^\infty \frac{1}{c_2^{\alpha k}} \\\nonumber 
        &\le C_\alpha m_\nu(v)^{\alpha+1},
    \end{align} where in \eqref{flowmeasurecond} we have used that $m_\nu$ is a flow measure. This proves \eqref{claim2} and concludes the proof of the first part of the theorem. \\
    We now prove that $ii)$ implies $i)$. Assume $ii)$. For every $y \in T$ let $a_{y}$ be a function supported in $\partial T_y$ with zero integral average and such that \begin{align}\label{taglia}\|a_{y}\|_{L^\infty}\le \frac{1}{m_\nu(y)}.
    \end{align}Define $$K_{a_y}(x,\omega)=\begin{cases}
        a_y(\omega) &\text{if $x=y$,} \\ 
        0&\text{otherwise.}
    \end{cases}$$ Clearly $K_{a_y}$ satisfies  $\eqref{A1}$ and $\eqref{A3}$ with $\alpha=1$ because $$|K_{a_y}(y,\omega)|=|a_y(\omega)|\le \|a_y\|_{L^\infty}\le \frac{1}{m_\nu(y)}=\frac{m_\nu(y)}{m_\nu(\omega \wedge y)^2} \qquad\forall \omega \in \partial T_y,$$ by \eqref{taglia} and the fact that  supp $a_y \subset \partial T_y$. Moreover, $K_{a_y}$ fulfills \eqref{A2} because
    \begin{align*}
        \sum_{x\in T} |K_{a_y}(x,\omega)|m_\nu(x)=|a_y(\omega)|m_\nu(y) \le 1 \qquad \forall \omega \in \partial T,
    \end{align*} again by invoking \eqref{taglia}. In particular, $C_{K_{a_{y}}}\leq 1$.
    By $ii)$ and the definition of $K_{a_y}$,
    \begin{align*}
        m_\nu(y) f(1,1)\ge \sum_{x \le y}|\cK_{a_y} b(x)|m_\nu(x)=m_\nu(y)\bigg|\int_{\partial T_y} a_y(\omega) b(\omega) \ \mathrm{d}\nu(\omega)\bigg|,
    \end{align*}
     from which it follows
     \begin{align}\label{sup:bmo}
    \sup_{y \in T} \bigg| \int_{\partial T}a_y(\omega) b(\omega) \ \mathrm{d}\nu(\omega)\bigg| \le f(1,1),\end{align}
    where $f$ is as in $ii)$.
     It is easy to deduce that \eqref{sup:bmo} implies that $b \in BMO.$ 
    Indeed,  for every $x \in T$ we have to prove that  
    \begin{align}\label{bmo}
       \frac{1}{m_\nu(x)} \int_{\partial T_x}|b(\omega)-b_{\partial T_x}| \ \mathrm{d}\nu(\omega) \le C
    \end{align} for some absolute constant $C>0$. We shall give a suitable definition of $a_y$ to get \eqref{bmo}. We will deal with the real case; if $b$ is complex-valued the argument can be slightly modified to obtain the same result. Define on $\partial T_y$ the function $a_y'$ by $$a_y'(\omega)=\begin{cases}
        1 &\text{if $b(\omega) \ge b_{\partial T_y}$,} \\ 
        -1 &\text{if $b(\omega) < b_{\partial T_y}$}.
    \end{cases}$$ Set $a_y(\omega)=\frac{1}{2m_\nu(y)}[a_y'(\omega)-(a_y')_{\partial T_y}]$ for every $\omega \in \partial T_y$. Observe that $a_y$ has zero average on $\partial T_y$ and satisfies \eqref{taglia} for every $x \in T$. Moreover,
\begin{align*}
    (a_y')_{\partial T_y}=\frac{1}{m_\nu(y)} \int_{\partial T_y} a_y'(\omega) \in [-1,1].
\end{align*}
  
   Using that $a_y$ has zero integral average on $\partial T_y$ and \eqref{sup:bmo}, it follows that 
\begin{align*}
   f(1,1) &\ge \bigg|\int_{\partial T_y} a_y(\omega)b(\omega) \ \mathrm{d}\nu(\omega) \bigg| \\
    &=\bigg|\int_{\partial T_y} a_y(\omega)\big(b(\omega)-b_{\partial T_y}\big) \ \mathrm{d}\nu(\omega) \bigg| \\ 
    &=\frac{1}{2m_\nu(y)}\bigg|\int_{\partial T_y}|b(\omega)-b_{\partial T_y}| \ \mathrm{d}\nu(\omega)-(a'_y)_{\partial T_y}\int_{\partial T_y}b(\omega)-b_{\partial T_y} \ \mathrm{d}\nu(\omega)\bigg| \\ 
    &= \frac{1}{2m_\nu(y)}\int_{\partial T_y}|b(\omega)-b_{\partial T_y}| \ \mathrm{d}\nu(\omega),
\end{align*}  because 
\begin{align*}
    \int_{\partial T_y}b(\omega)-b_{\partial T_y} \ \mathrm{d}\nu(\omega)=0.
\end{align*}
    This implies \eqref{bmo} and thus $b \in BMO.$
\end{proof}
\begin{remark}
    Note that   $\eqref{A3}$ does not imply   $\eqref{A2}$. Indeed, if $K(x,\omega)=\displaystyle\frac{m_\nu(x)^{\alpha}}{m_\nu(\omega \wedge x)^{\alpha+1}} $ for some $\alpha>0$ then $K$ clearly satisfies \eqref{A3} but
\begin{align*}
 \sum_{x \in T}\frac{m_\nu(x)^{\alpha}}{m_\nu(\omega \wedge x)^{\alpha+1}} m_\nu(x) \ge  & \sum_{x \in(\omega,\omega_*)}1 =\infty \qquad \forall \omega \in \partial T,
 \end{align*} so $K$ does not satisfy \eqref{A2}. 
 \end{remark}
 \begin{example}We provide some examples of operators in $\cO$. Given  $\delta>0$ and $\alpha>0$, assume that a function $K_\delta$ satisfies 
 \begin{align}\label{assumption4}
|K_\delta(x,\omega)|\le \frac{m_\nu(x)^\alpha}{m_\nu(\omega \wedge x)^{\alpha+1}}\min\bigg\{ \frac{1}{m_\nu(\omega \wedge x)}, m_\nu(\omega \wedge x)\bigg\}^{1+\delta}.
 \end{align} It is clear that 
 \begin{align*}
  |K_\delta(x,\omega)|\le \frac{m_\nu(x)^\alpha}{m_\nu(\omega \wedge x)^{\alpha+1}}.
 \end{align*} Moreover, recalling the map $\Phi$ in Definition \ref{def:Phi} and that $\omega \wedge x=\Phi(\omega,k)$ for every $x \in T_{\Phi(\omega,k)}\setminus T_{\Phi(\omega,k-1)}$ and $k \in \mathbb N$,
 \begin{align}
     \nonumber\sum_{x \in T}&|K_\delta(x,\omega)| m_\nu(x) \\ \nonumber&\le \sum_{k=-\infty}^{\infty}\sum_{\substack{x \in T_{\Phi(\omega,k)}, \\ \nonumber x \not \in T_{\Phi(\omega,k-1)}}} \frac{m_\nu(x)^{\alpha+1}}{m_\nu(\Phi(\omega,k))^{\alpha+1}}\min\bigg\{\frac{1}{m_\nu(\Phi(\omega,k))},m_\nu(\Phi(\omega,k))\bigg\}^{1+\delta} \\ \nonumber &\le \sum_{k=-\infty}^{\infty}\min\bigg\{\frac{1}{m_\nu(\Phi(\omega,k))^{\alpha+2+\delta}},m_\nu(\Phi(\omega,k))^{\delta-\alpha}\bigg\}\sum_{x \in T_{\Phi(\omega,k)}} m_\nu(x)^{\alpha+1} \\ 
     &\label{power-flow} \le C \sum_{k=-\infty}^{\infty}\min\bigg\{\frac{1}{m_\nu(\Phi(\omega,k))^{2+\delta+\alpha}},m_\nu(\Phi(\omega,k))^{\delta-\alpha}\bigg\}m_\nu(\Phi(\omega,k))^{\alpha+1} \\ \nonumber
     &=C\sum_{k=-\infty}^{\infty}\min\bigg\{\frac{1}{m_\nu(\Phi(\omega,k))^{\delta+1}},m_\nu(\Phi(\omega,k))^{\delta+1}\bigg\}\\ \nonumber &=C\bigg(\sum_{k \ : \ m_\nu(\Phi(\omega,k)) \ge 1} \frac{1}{m_\nu(\Phi(\omega,k))^{\delta+1}}+\sum_{k\ : \ m_\nu(\Phi(\omega,k))<1} m_\nu(\Phi(\omega,k))^{\delta+1}\bigg)
 \end{align} where  \eqref{power-flow} is proved as in \eqref{claim2}. We next claim that 
 \begin{align*}
\sum_{k \ : \ m_\nu(\Phi(\omega,k)) \ge 1} \frac{1}{m_\nu(\Phi(\omega,k))^{\delta+1}}+\sum_{k\ : \ m_\nu(\Phi(\omega,k))<1} m_\nu(\Phi(\omega,k))^{\delta+1}\le C_\delta.
 \end{align*}
 Indeed, set $k_0$ as the biggest integer such that $m_\nu(\Phi(\omega,k_0)) < 1$, which exists by \eqref{constants-locdoub}.  
 By \eqref{constants-locdoub} we deduce that
 \begin{align*}
 \sum_{k \ : \ m_\nu(\Phi(\omega,k)) \ge 1} \frac{1}{m_\nu(\Phi(\omega,k))^{\delta+1}} &\le \frac{1}{m_\nu(\Phi(\omega,k_0+1))^{\delta+1}} \sum_{k \ge k_0+1} \frac{1}{c_2^{(k-k_0-1)(\delta+1)}} \\ &\le C \frac{1}{m_\nu(\Phi(\omega,k_0+1))^{\delta+1}} \le C_\delta,\\ 
     \sum_{k\ : \ m_\nu(\Phi(\omega,k))<1} m_\nu(\Phi(\omega,k))^{\delta+1} &\le m_\nu(\Phi(\omega,k_0))^{\delta+1}\sum_{k \ge k_0} \frac{1}{c_2^{(k-k_0
)(\delta+1)}}\\ &\le C m_\nu(\Phi(\omega,k_0))^{\delta+1} \le C_\delta. 
 \end{align*} In summary, we showed that if $K_\delta$ satisfies \eqref{assumption4} then it also satisfies \eqref{A2}. We conclude by showing that it is possible to construct a $K_\delta$ satisfying \eqref{assumption4} that also fulfills  \eqref{A1}. Indeed, for every fixed $x \in T$, the map $\partial T\ni \omega \mapsto K_\delta(x,\omega)$ is integrable because
 \begin{align*}
     \int_{\partial T} |K_\delta(x,\omega)| \ \mathrm{d}\nu(\omega)&=\int_{\partial T_x}|K_\delta(x,\omega)| \ \mathrm{d}\nu(\omega)+\sum_{k=1}^\infty\int_{\partial T_{p^k(x)}\setminus T_{p^{k-1}(x)}}|K_\delta(x,\omega)| \ \mathrm{d}\nu(\omega)\\ &\le\frac{1}{m_\nu(x)}m_\nu(x)^{2+\delta}\\ &+\sum_{k=1}^\infty \frac{m_\nu(x)^\alpha}{m_\nu(p^k(x))^{\alpha}}\min\bigg\{ \frac{1}{m_\nu(p^k(x))}, m_\nu(p^k(x))\bigg\}^{1+\delta} \\
     &\le C_{x,\delta}.
 \end{align*} Next, we have to construct a kernel that also satisfies the zero integral condition \eqref{A1}. For every $x \in T$  we set $c_x(\omega)=c_x(k)$ if $\omega \in \partial T_{p^k(x)}\setminus \partial T_{p^{k-1}(x)}$ if $k\ge 1$ and $c_x(\omega)=c_x(0)$ if $\omega \in \partial T_x$ where $\{c_x(k)\}_{k \in \mathbb N} \in \ell^2(\mathbb N)$ is to be chosen. We also assume that $\{c_x(k)\}_{k \in \mathbb N}$ belongs to the unit ball of $\ell^\infty(\mathbb N)$. Then we set 
 \begin{align*} K_\delta(x,\omega)=c_x(\omega)\frac{m_\nu(x)^\alpha}{m_\nu(\omega \wedge x)^{\alpha+1}}\min\bigg\{ \frac{1}{m_\nu(\omega \wedge x)}, m_\nu(\omega \wedge x)\bigg\}^{1+\delta}. 
 \end{align*} For notational convenience, for every $k\ge 1$ we set 
 \begin{align*}
     d_x(k)=\int_{\partial T_{p^k(x)}\setminus \partial T_{p^{k-1}(x)}} \frac{m_\nu(x)^\alpha}{m_\nu(\omega \wedge x)^{\alpha+1}}\min\bigg\{ \frac{1}{m_\nu(\omega \wedge x)}, m_\nu(\omega \wedge x)\bigg\}^{1+\delta} \ \mathrm{d}\nu(\omega),
 \end{align*} and similarly 
 \begin{align*}
     d_x(0)=\int_{\partial T_{x}} \frac{m_\nu(x)^\alpha}{m_\nu(\omega \wedge x)^{\alpha+1}}\min\bigg\{ \frac{1}{m_\nu(\omega \wedge x)}, m_\nu(\omega \wedge x)\bigg\}^{1+\delta} \ \mathrm{d}\nu(\omega).
 \end{align*}

 Then
 \begin{align*}
     \int_{\partial T}K_\delta(x,\omega) \ \mathrm{d}\nu(\omega)&=\sum_{k \ge 0 }c_x(k)d_x(k)=\langle c_x, d_x\rangle_{\ell^2(\mathbb N)}.
 \end{align*} Since $\{d_x(k)\}_{k \in \mathbb N} \in \ell^1(\mathbb N) \subset \ell^2(\mathbb N)$, we can choose any $c_x \in \{d_x(k)\}_{k \in \mathbb N}^\perp \subset \ell^\infty(\mathbb N)$  where the orthogonal complement is taken with respect to the $\ell^2(\mathbb N)$ inner product. Since such a $c_x \in \ell^\infty(\mathbb N)$ we can assume without loss of generality that $\|c_x\|_{\ell^\infty(\mathbb N)}=1$. We finally conclude that  for such a choice
 \begin{align*}
       \int_{\partial T}K_\delta(x,\omega) \ \mathrm{d}\nu(\omega)&=0,
 \end{align*} and thus $K_\delta$ fulfills   \eqref{A1},\eqref{A2} and \eqref{A3} of Definition \ref{def:O}.
\end{example}

 \bibliographystyle{plain}
{\small
\bibliography{references}}
\end{document}